\documentclass{amsart}
\usepackage{amssymb}
\usepackage{amsthm}
\usepackage[utf8]{inputenc}
\newcommand\mut[1]{\ignorespaces}
\usepackage[usenames]{color}
\usepackage{xcolor}

\newtheorem{thm}{Theorem}

\newtheorem{lem}[thm]{Lemma}

\newtheorem{prop}[thm]{Proposition}

\newtheorem{rem}[thm]{Remark}

\def\ap{\alpha}
\def\tm#1{\left(\text{mod }#1\right)}
\def\mQ{\mathbb Q}

\keywords{Additive number theory, Arithmetic functions, divisor function}
\title{Every integer can be written as a square plus a squarefree}

\author{Jorge Jim\'enez Urroz}
\thanks{Partially supported by grant PID2019-110224RB-I00 of MICINN}
\address{Departamento de Matem\'aticas \\
Universitat Polit\'ecnica Catalunya\\ Edificio C3 - Campus Nord UPC, Carrer de Jordi Girona 1-3, 08034 Barcelona, Spain}
\email{jorge.urroz@upc.edu}

\makeindex

\date{}
\begin{document}

\begin{abstract}{In the paper we can prove that every integer can be written as the sum of  two integers, one perfect square and one squarefree. We also establish the asympotic formula for the number of representations of an integer in this form. The result is deeply related with the divisor function. In the course of our study we get an independent result about it. Concretely we are able to deduce a  new upper bound for the divisor function fully explicit.}
\end{abstract}
\maketitle

\section{Introduction}
In his paper \cite{ester}, (afterwards generalized in different ways,  \cite{ester}, \cite{ester2},\cite{ester3},\cite{lineve}, \cite{p}, \cite{ro},\cite{ss}), Estermann  considers the problem of finding the correct order of magnitude of the function $r(n)$ which counts the number of representations of 
an integer $n$ as a sum of a square plus a squarefree.  The method is quite involved and
has a main disadvantage, which is that even after giving the main term of the asymptotic,
it is not possible to deduce the natural question wether any positive integer $n$ can indeed 
be written as a sum of a square plus squarefree. In the present note, we present a very simple
method, completely explicit, to answer this question. The first result is
\begin{thm}\label{sqfree} Every positive integer $n$ is representable as
$$
n=\square+\,\,^{_{{\text{\circle{9}{}}}}}
$$ 
for some square $\square$ and some positive squarefree integer \, $^{_{{\text{\circle{9}{}}}}}$.
\end{thm}
But, in fact, we can also establish the correct order of magnitude of $r(n)$ in a 
more direct and simple way than the one used in \cite{ester}. Let $g$ be the completely multiplicative function which, at prime $p$, counts the number of solutions to the congruence class $x^{2}\equiv n\pmod{p^2}$.

\begin{thm}\label{repsqfree} Let $r(n)$ the number of primitive representations of $n$ as a square plus squarefree. Then
$$
r(n)=n^{1/2}\prod_{p}\left(1-\frac{g(p)}{p^2}\right)+E(n),
$$
where
$$
|E(n)|<\frac{24}{(\log n)^{\frac34}}n^{\frac{\log 2+0.342}{\log\log n}}.
$$
\end{thm}
It is interesting to note that, even though Theorem \ref{repsqfree} is absolutely explicit,
only  provides a proof of Theorem \ref{sqfree} for $n$ big enough. In fact, one can only get positivity of $r(n)$ for values of $n\ge 10^{440}$. This is clearly imposible to verify by means of computers, and might be the case that we will never reach this bound, independently of the speed of the cpu. 
The main issue is making the error term explicit, and as small as possible. However, the error term depends on the multiplicative function $g(\cdot)$ which, for primes $p|n$ is as big as
$p$. 

\

If one wants to avoid those primes $p|n$, so the $g$ function remains absolutely bounded, one
ends with another natural problem which is study the behaviour of $r^*(n)$, the number of primitive representations of $n$ as a square plus squarefree, i.e. representations in which both $\square$ and  \, $^{{{\text{\circle{9}{}}}}}$ are coprime with $n$.  Again one can also get the asymptotic formula in a very similar way.  It is possible to prove
\begin{thm}\label{primrepsqfree} Let $r^*(n)$ the number of primitive representations of $n$ as a square plus squarefree. Then
$$
r^*(n)=n^{1/2}\prod_{p|n}\left(1-\frac1p\right)\prod_{p\nmid n}\left(1-\frac{g(p)}{p^2}\right)+ E^*(n),
$$
where
$$
E^*(n)<\frac{24}{(\log n)^{\frac34}}n^{\frac{\log 2+0.342}{\log\log n}}.
$$
\end{thm}

No matter which of the two functions we are considering,  we  see that in a natural way the error term depends on the divisor function $\tau(n)$, and this is what makes untractable the proof of Theorem \ref{sqfree}. Although it is known that $\tau(n)=O(n^{\delta})$ for any positive $\delta$, the constant in front of this estimation grows double exponentially with $\delta$ which causes the enormous values of $n$ for which the error term starts to be negligible.

\

To overtake this problem we need a fully explicit upper bound for the divisor function. This problem has already been consider in the literature since  $1907$ year in which Wigert gets in \cite{wigert} the correct order of magnitude for the divisor function. Concretely he is able to prove to things. Frist that  $\tau(n)<2^{\frac{\log n}{\log\log n}(1+\varepsilon)}$ is valid for any $\varepsilon>0$ and all $n$ suffficiently large depending on $\varepsilon$,  and also that $\tau(n)>2^{\frac{\log n}{\log\log n}(1-\varepsilon}$ is valid for any $\varepsilon$  and infinitely many $n$. While for the first result the author uses elementary techniques, for the second needs the knowledge of the prime number theorem. Then Ramanujan,  (see \cite{rama} or  \cite [p.115]{r}) using his highly composite numbers, deduce the bounds elementary, and gets
$$
\limsup\frac{\log(\tau(n))\log\log n}{\log 2\log n}=1,
$$
together with a very complete account of the divisor function. This bound now is part of any elementary course in number theory and  can be found for example in \cite[ch.18]{hawri}).  However, Ramanujan misses to get an explicit bound valid for any integer $n$ and it is in   \cite{niro} were the authors get 
$$
\frac{\log(\tau(n))\log\log n}{\log 2\log n}\le 1.5379.
$$
valid for any $n$. We could use this to continue our study. However,  we dedicate Section \ref{sec:divisor} to find a completely explicit upper bound for the divisor function, in a more direct way than in  \cite{niro}, but also more complete, in the following sense. We prove 
\begin{thm}\label{teo:tau} For every  $\varepsilon>0$, and every $n>e^{e^{\frac{14}{\sqrt{1+28\varepsilon}-1}}}$ we have 
$$
\tau(n)<n^{\frac{\log 2+\varepsilon}{\log\log n}}.
$$
In particular, for any $n\ge 3$
$$
\tau(n)<n^{\frac{3}{\log\log(n)}}.
$$
\end{thm}

\begin{rem}The second inequality is not best possible, as it is shown in \cite{niro}. However, we include it to show its simplicity in contrast with the results in \cite{niro}.
\end{rem}
\

The first  application we can give  of this theorem is  simply ending the proofs of  Theorems \ref{sqfree}, \ref{repsqfree} and \ref{primrepsqfree}. However, the applications of these type of upper bounds are vast, (see \cite{kole}, \cite{r}, or \cite{tao}) and we expect to show some others in the near future. 

\section{Starting the proof of Theorems \ref{repsqfree} and \ref{primrepsqfree}.}

The proofs of Theorems \ref{repsqfree} and \ref{primrepsqfree} are very similar so
we will just do the first one and include the main difference to get the second one, leaving
the details to the interested reader.  

\

We will use standard notations as $[x]$ for the integer part of $x$, $\{x\}=x-[x]$,  the Landau $O$ notation, or $p^e||n$ if $p^e$ is the highest power of a prime $p$ dividing $n$. Also a $\mathfrak b$ above a sum will restrict the sum to squarefree integers, $\pi(x)$ counts the number of primes up to $x$ and $\theta(x)=\sum_{p\le x}\log p$.

\

Consider $S(n)=\{0\le x\le n^{1/2}:n-x^{2} \text{ is squarefree}\}$. Then, $r(n)=|S(n)|$ and
\begin{eqnarray*}
r(n)&=&\sum_{0\le x\le n^{1/2}}\sum_{d^2|(n-x^{2})}\mu(d)=\sum_{d\le \sqrt{n}}\mu(d)\sum_{{0\le x\le n^{1/2}}\atop{d^2|n-x^{2}}}1\\
&=&\sum_{d\le y}\mu(d)\sum_{{0\le x\le n^{1/2}}\atop{d^2|n-x^{2}}}1+
\sum_{y<d\le \sqrt{n}}\mu(d)\sum_{{0\le x\le n^{1/2}}\atop{d^2|n-x^{2}}}1\\
&=&S_1+S_2,
\end{eqnarray*}
for any $0\le y<\sqrt n$. Given any integer $n$, let $g(\cdot)$ be the completely multiplicative function with value at primes given by the number of solutions to the equation $x^{2}\equiv n \tm{p^2}$ . Observe that  $g(2)=0$ if $n\equiv 2,3 \tm{4}$ and $2$ if $n\equiv 0,1 \tm {4}$, and for $p$ odd we have  $g(p)=0$ if $n$ is not a square mod $p$, or if $p||n$, $g(p)=p$ if $p^2|n$ and finally $g(p)=2$ if $n$ is a nonzero square modulo $p$. Then, first note that the sum can be restricted to squarefree integers $d$. Further note that for any squarefree $d$,  on each interval of lenght $d^2$ there are exactly $g(d)$ solutions so we get
$$
\sum_{\substack{0\le x\le n^{1/2}\\ d^2|n-x^{2}}}1=n^{\frac12}\frac{g(d)}{d^2}+R(d)
$$
where
$$
|R(d)|\le g(d), 
$$
and so
\begin{eqnarray}\label{eq:s1}
S_1&=&n^{1/2}\sum_{d\le y}\mu(d)\frac{g(d)}{d^2}+\sum_{d\le y}\mu(d)R(d)\nonumber \\
&=&n^{1/2}\prod_p\left(1-\frac{g(p)}{p^2}\right)+n^{1/2}\sum_{d> y}\mu(d)\frac{g(d)}{d^2}+\sum_{d\le y}\mu(d)R(d)
\end{eqnarray}

%As usual, $(d,n)$ is the greatest common divisor of the integers $d,n$ and   $\omega(n)$ %denotes the number of distinct prime factors of $n$. 
The first part
in (\ref{eq:s1}) will be the main term of $r(n)$. Then, we need to bound the error term. We start with the first part $\sum_{d> y}\mu(d)\frac{g(d)}{d^2}$.
Now, for any $z>20$ we have
\begin{eqnarray*}
\sum_{d>z}^{\kern14pt\mathfrak b}\frac{2^{\omega(d)}}{d^2}&\le&\sum_{d>z}\frac{1}{d^2}\sum_{k|d}1=
\sum_{k}\frac{1}{k^2}\sum_{d>z/k}\frac{1}{d^2}\\
&\le& \zeta(2)\sum_{k\ge z-2}\frac{1}{k^2}
+
\sum_{k<z-2}\frac{1}{k(z-k)}\\
&=&\zeta(2)\sum_{k\ge z-2}\frac{1}{k^2}
+
\frac1z\left(\sum_{k<z-2}\frac{1}{k}+\sum_{k<z-2}\frac{1}{z-k}\right)
\\
&<&
\frac{\zeta(2)}{z-3} +\frac{1+2\log z}{z}<
\frac{3\log z}{z},
\end{eqnarray*}
meanwhile if $z\le 20$, then
$$
\sum_{d>z}^{\kern14pt\mathfrak b}\frac{2^{\omega(d)}}{d^2}\le\prod_p\left(1+\frac{2}{p^2}\right)<2.2
$$ 
and so
\begin{eqnarray*}
\left|\sum_{d>y}\mu(d)\frac{g(d)}{d^2}\right|&\le& \sum_{r|n_\mathfrak s}^{\kern14pt\mathfrak b}\frac{1}{r} \sum_{d>y/r}^{\kern14pt\mathfrak b}\frac{2^{\omega(d)}}{d^2}\le \sum_{\substack{r|n_\mathfrak s\\r< \frac y{20}}}^{\kern14pt\mathfrak b}\frac{1}{r} \sum_{d>y/r}^{\kern14pt\mathfrak b}\frac{2^{\omega(d)}}{d^2}+\sum_{\substack{r|n_\mathfrak s\\r\ge \frac y{20}}}^{\kern14pt\mathfrak b}\frac{1}{r} \sum_{d>y/r}^{\kern14pt\mathfrak b}\frac{2^{\omega(d)}}{d^2}\\
&\le& 
\left(\frac{3\log y+44}{y}\right)\tau(n_\mathfrak s).
\end{eqnarray*}
Moreover, the main term in $r(n)/n^{1/2}$ is a convergent product for any given  $n$ and verifies, by the
definition of $g(p)$,  
\begin{equation}\label{eq:gipi}
\prod_{p}\left(1-\frac{g(p)}{p^2}\right)\ge\prod_{p\nmid n}\left(1-\frac{2}{p^2}\right)\prod_{p|n_\mathfrak s}\left(1-\frac{1}{p}\right)>\frac{\prod_{p}\left(1-\frac{2}{p^2}\right)}{e^{\gamma}\log\log n+\frac{3}{\log\log n}}.
\end{equation}
For the last inequality see for example \cite[p.114]{l}, or \cite{clark}.

\

For the second error term in (\ref{eq:s1}) we have the upper bound
\begin{eqnarray}\label{error}
\sum_{d\le y}^{\kern14pt\mathfrak b}g(d)\le y\sum_{d\le y}^{\kern14pt\mathfrak b}\frac{g(d)}{d}\le y\prod_{p\le y}\left(1+\frac{g(p)}{p}\right)\le ye^{2\sum_{p\le y}\frac1p}\prod_{p^2|n}2.
\end{eqnarray}
We have used the previuos identities for $g(p)$, with the trivial bounds  $2\le e$ and $\log(1+x)\le x$ for $0<x<1$. We will now need the following proposition.
\begin{prop}\label{eq:inverse}  For every $x\ge 2$ we have
\begin{equation}\label{eq:inverse2}
\sum_{p\le x}\frac1p\le \log\log x +3-\log\log 2.
\end{equation}
\end{prop}
We start with the following inequality
 \begin{equation}\label{sum}
x\log x\ge \sum_{n\le x}\log n\ge\sum_{p\le x}\log p\left[\frac xp\right]\ge x\sum_{p\le x}\frac{\log p}p-\sum_{p\le x}\log p,
\end{equation}
which is true since $\log p$ will appear, at least once, for each multiple of $p$ less than $x$. 
Now we prove the following lemma
\begin{lem}\label{erdos} Let $\theta(x)=\sum_{p\le x}\log p$. For any $x\ge 8$  we have
$$
\theta(x)< (x-4)\log 4.
$$
\end{lem}
We prove it by induction.   Let us call
$p_n$ the $n$-th prime number. The result  is clearly true for $8\le x<17$, and for 
$p_n\le x< p_{n+1}$, we have $\theta(x)=\theta(p_n)$, so we suppose it is true for $x< p_{n+1}$, and want
to prove it for $x=p_{n+1}$. However 
$$
\prod_{\frac{p_{n+1}+1}{2}<p\le p_{n+1}}p\le \binom{p_{n+1}}{\frac{p_{n+1}+1}{2}}<2^{p_{n+1}-1},
$$
so
\begin{eqnarray*}
\theta(p_{n+1})&=&\sum_{\frac{p_{n+1}+1}{2}<p\le p_{n+1}}\log p+\sum_{p\le\frac{p_{n+1}+1}{2}}\log p\\ 
&\le&  (p_{n+1}-1)\log 2+\log 4\frac{p_{n+1}+1}{2}-4\log 4\\
&=& (p_{n+1}-4)\log 4.
\end{eqnarray*}
{\bf Remark:}
This same proof should serve to deduce that $\theta(x)< (x-c)\log 4$
for every $c$, and $x$ large enough depending on $c$, by just bounding the number of primes $p$ such that $2p-1$ is also prime. Also, observe that this is basically the same old idea
of Thcebychev, although Lemma \ref{erdos} is a little bit more accurate.

\

Plugging  Lemma \ref{erdos} into (\ref{sum}) we get 
\begin{equation}\label{eq:primeinvers}
\sum_{p\le x}\frac{\log p}p\le \log x+\log 4,
\end{equation}
for $x\ge 8$, but for $2\le x\le 8$ the previous inequality is trivial.
On the other hand, if we let $\hat \theta(x)$ be the sum on the left of (\ref{eq:primeinvers})  we have, by partial summation,
$$
\sum_{p\le x}\frac1p=\frac1{\log x}\hat \theta(x)+\int_{2}^{x}\frac{\hat \theta(t)}{t(\log t)^2}dt\le
\log\log x +3-\log\log 2,
$$
as we wanted. 

We just have to use Proposition \ref{eq:inverse} in (\ref{error}) to get 
\begin{equation}\label{rd}
\sum_{d\le y}^{\kern14pt\mathfrak b}|R(d)|\le \frac{e^6}{(\log 2)^2}y(\log y)^2\tau(n_s),
\end{equation}
where $n=n_{\mathfrak s}^2n_{\mathfrak b}$ with $n_{\mathfrak b}$ squarefree. 

\

We now treat $S_2$. It is straighforward to see that 
\begin{equation}\label{sru}
|S_2|\le \#\{(k,x,d): k,\text{ squarefree } d>y, \text{ and }n=x^2+kd^2\}\le \frac{6n}{y^2}\tau(n).
\end{equation}
For the second inequality  observe that  we have at most $n/y^2$ such $k$ and, for each one of them, there are at most $\tau(n)$ pairs $(x,d)$ since every solution of $n=x^{2}+d^2k$ corresponds to a factorization of $n$ in ideals in the field 
$\mQ(\sqrt{-k})$. Indeed, let $\mathfrak P$ be a  prime above $p$ and let the factorization of $n$ into ideals be 
$$
n=\prod_{p}p^{e_p}=\prod_{\substack{p|n\\ p\text { ramifie}}}\mathfrak P^{2e_p}\times
\prod_{\substack{p|n\\ p\text { non split}}}p^{e_p}\times\prod_{\substack{p|n\\ p\text { split}}}
\mathfrak P^{e_p}\bar{\mathfrak P}^{e_p},
$$
where $\bar{\mathfrak I}$ denotes the conjugate ideal of $\mathfrak I$, and let $n=x^2+kd^2=(x+\sqrt{-k}d)(x-\sqrt{-k}d)$ be a representation of $n$. Then, the norm $N(x+\sqrt{-k}d)=N(x-\sqrt{-k}d)=n$, so $n$ is representable only if $e_p$ is even for non split primes factors of $n$ and, in this case, the only possible representations come  from different selections of the exponents in the prime ideals of splitting primes, and its conjugates
$$
x+\sqrt{-k}d=\prod_{\substack{p|n\\ p\text { ramifie}}}\mathfrak P^{e_p}\times
\prod_{\substack{p|n\\ p\text { non split}}}p^{\frac{e_p}2}\times\prod_{\substack{p|n\\ p\text { split}}}
\mathfrak P^{\alpha_p}\bar{\mathfrak P}^{e_p-\alpha_p},
$$
for any $0\le \alpha_i\le e_i$,  $i=1,\dots, r$. Hence we have 
$$
\prod_{\substack{p|n\\ p\text { split}}}(1+e_i)\le \tau(n)
$$ 
possible selections of exponents. From here, taking into account the units of the quadratic field, which are at most $6$, we get our bound for $S_2$.

\

Using the trivial inequality  $\tau(n)=O(n^{\delta})<C(\delta) n^{\delta}$ for some constant $C(\delta)$ we get,  plugging the previous estimates in (\ref{rd}) and (\ref{sru}) %and choosing  $y= \frac{(6(\log 2)^2)^{1/3}}{e^2}\frac{n^{1/3+\delta/6}}{(\log n)^{2/3}}$,
$$
r(n)=n^{1/2}\prod_{p}\left(1-\frac{g(p)}{p^2}\right)+E(n),
$$
where
$$
E(n)<\left(\frac{6n}{y^2}+\frac{e^6}{(\log 2)^2}y(\log y)^2+\frac{3\log y+44}{y}n^{1/2}\right)\tau(n).
$$
Taking $y=\frac{n^{1/3}}{4.54(\log n)^{2/3}}$, and using the bound for the divisor function  we see that
$$
E(n)<160C(\delta)n^{1/3+\delta}(\log n)^{4/3},
$$
 for any $n\ge 8100$.
%$$
%|E(n)|\le  C(\delta)(\log n)^{4/3}n^{1/3+2\delta/3}
%\left(\frac{25(48)^{1/3} e^4}{144(\log 2)^{4/3}}+\frac{5}{3n^{2/3-\delta/6}}\right).
%$$

\section{The divisor function}\label{sec:divisor}
However, in order to prove the theorem, we need an explicit upper bound for the divisor function.  This is the content of the next lemma. 
\begin{prop}\label{prop:tau} For any integer $n$, and any $0<\delta<1$ we have  
$$
\tau(n)< e^{H(\delta)} n^{\delta}.
$$
where $H(\delta)=\delta^2\frac {2^{\frac1\delta}}{(\log 2)^2}+7\delta^3\frac {2^{\frac1\delta}}{(\log 2)^3}$.
\end{prop}

\begin{proof} The result is a consequence of the upper bound given by  Ramanujan in  \cite[p.113]{r}, $\tau(n)\le C(\delta)n^{\delta}$, where
$$
C(\delta)=\prod_{k\ge 1}\left(\frac{_{k+1}}{^k}\right)^{\pi((1+1/k)^{1/\delta})}
e^{-\delta\theta((1+1/k)^{1/\delta})}.
$$
Now, we see that
\begin{eqnarray*}
\log (C(\delta))&=&\delta\sum_{k\le \frac{1}{2^{\delta}-1}}\sum_{p\le (1+\frac1k)^{1/\delta}}\left(\log\left(1+\frac1k\right)^{1/\delta}-\log p\right)\\
&=&\delta\sum_{k\le \frac{1}{2^{\delta}-1}}\sum_{p\le (1+\frac1k)^{1/\delta}}\int_p^{(1+1/k)^{1/\delta}}\frac{1}{t}dt
=\delta\sum_{k\le \frac{1}{2^{\delta}-1}}\int_2^{(1+1/k)^{1/\delta}}\frac{\pi(t)}{t}dt,
\end{eqnarray*}
where, for the last identity, we use that  the integral between two consecutive primes $\int_{p_n}^{p_{n+1}}\frac{1}{t}dt$ appears exactly $\pi(p_n)$ times in the sum over primes. Now, the trivial bound $\pi(t)\le t$ already gives
\begin{eqnarray*}
\log (C(\delta))&<&\delta\sum_{k\le \frac{1}{2^{\delta}-1}} (1+1/k)^{1/\delta}=\delta2^{1/\delta}+\delta\sum_{2\le k\le \frac{1}{2^{\delta}-1}} (1+1/k)^{1/\delta}
\\
&<&\delta2^{1/\delta}+\delta\left(\frac32\right)^{1/\delta}\frac{1}{2^{\delta}-1}<\delta2^{1/\delta}+\left(\frac32\right)^{1/\delta}\log 2<
2\delta 2^{1/\delta},
\end{eqnarray*}
where the last inequalities are consequence of  the mean value theorem, and the bound $\log 2<\delta\left(\frac43\right)^{1/\delta}$ valid for any $\delta>0$.

\

But this is not the best we can do. To improve it, we will use a slight generalization of the well known partial summation lemma
\begin{lem}\label{lem:summa} 
Let $M,N$ integers, $f(x)$ a continuous function, differentiable at $[M,N]$ except in a set $\frak C$ with finitely many elements, $\{a_n\}_{n\in\mathbb N}\subset\mathbb R$ any sequence and consider for any $k\ge 0$, $S(k)=\sum_{n\le k}a_n$.  Then
$$
\sum_{M< n\le N}f(n)a_n=f(N)S(N)-f(M)S(M)-\int_M^NF(t)S(t)dt,
$$
where $F(t)=f'(t)$ for any $t\notin \frak C$ and $F(t)=0$ for $t\in \frak C$.
\end{lem}
\begin{proof} First note that it is enough to prove it for $M=0$, since the general case comes by substracting the sum up to $M$ from the sum up to $N$.  Now we start, as in the standard summation by parts, by the identity
\begin{eqnarray}\label{eq:summa}
\sum_{n\le N}f(n)a_n&=&\sum_{n\le N}f(n)(S(n)-S(n-1))=\sum_{n\le N}f(n)S(n)-\sum_{n\le N}f(n)S(n-1)\nonumber\\
&=&f(N)S(N)-\sum_{n\le N-1}S(n)(f(n+1)-f(n)).
\end{eqnarray}
Now let 
$$
\{n=x_{1,n}<x_{2,n}<\dots<x_{{l_n,n}}=n+1\}=\frak C\cap[n,n+1]. 
$$
Then
\begin{eqnarray*}
&&\sum_{n\le N-1}S(n)(f(n+1)-f(n))=\sum_{n\le N-1}S(n)\sum_{k=1}^{l_n-1}(f(x_{k+1,n})-f(x_{k,n}))\\
&&=\sum_{n\le N-1}S(n)\sum_{k=1}^{l_n-1}\int_{x_{k,n}}^{x_{k+1,n}}f'(t)dt=\sum_{n\le N-1}S(n)\sum_{k=1}^{l_n-1}\int_{x_{k,n}}^{x_{k+1,n}}F(t)dt\\
&&=\sum_{n\le N-1}S(n)\int_{n}^{n+1}F(t)dt=\sum_{n\le N-1}\int_{n}^{n+1}F(t)S(t)dt=\int_{0}^{N}F(t)S(t)dt
\end{eqnarray*}
and pluging it into (\ref{eq:summa}) finishes the proof of the lemma.
\end{proof}
To prove Proposition \ref{prop:tau} we use  Lemma \ref{lem:summa} with $g(x)=f\circ h(x)$,
$f(x)=\int_2^x\frac{\pi(t)}{t}dt$ and $h(x)=\left(1+\frac1x\right)^{\frac1\delta}$. Observe that $h'(x)=-\frac1\delta\frac{h(x)}{x(x+1)}$, and $f(x)$ is a continuous function, differentiable at any $x\ne p$ a prime number,  with $f'(x)=\frac{\pi(x)}{x}$. Then by the chain rule $g'(x)=-\frac1\delta\frac{h(x)}{x(x+1)}\frac{\pi(h(x))}{h(x)}=-\frac1\delta\frac{\pi(h(x))}{x(x+1)}$ and 
 taking $N=\left[\frac{1}{2^{\delta}-1}\right]$, $M=1$, we get 
$$
\delta\sum_{k\le N}\int_2^{(1+1/k)^{1/\delta}}\frac{\pi(t)}{t}dt=\delta\sum_{k\le N}g(k)=\delta g(N)N+\int_1^{N}\frac{\pi(h(t))}{t(t+1)}[t]dt
$$
To bound the first term, we observe that $g(\frac1{2^\delta-1})=0$, and hence for some number $\xi\in[N,\frac1{2^\delta-1}]$ we have
\begin{eqnarray*}
g(N)&=&\left(N-\frac1{2^\delta-1}\right)g'(\xi)=\frac1\delta\left(\frac1{2^\delta-1}-N\right)\frac{\pi(h(\xi))}{\xi(\xi+1)}\\
&<&\frac1\delta\left(\frac1{2^\delta-1}-N\right)\frac{\pi(h(N))}{N(N+1)}
\end{eqnarray*}
since $h(x)$ is a decreasing function. Now, the trivial bound $\log(1+x)<x$, valid for any $x>0$ gives
$h(N)<e^{\frac{1}{N\delta}}$. Also, 
$$
\frac{1}{N\delta}=\frac{2^\delta-1}{\delta}\left(\frac1{1-(2\delta-1)\{\frac1{2^\delta-1}\}}\right)<\frac{2^\delta-1}{\delta}\left(\frac1{2-2\delta}\right)<\frac{1}{1-\delta}.
$$
For the last inequality we just use the mean value theorem for the function $f(x)=2^x$ twice. In the interval $[0,\delta]$ and also in  $[\delta, 1]$. On the other hand, for any positive $x>1$ we have
$$
\frac{x-[x]}{[x]}=\frac{\{x\}}{x}\left(\frac1{1-\frac{\{x\}}{x}}\right)<\frac1{x-1},
$$
and taking $x=\frac1{2^\delta-1}$,  we get
$$
\frac{\frac1{2^\delta-1}-N}{N}<\frac{2^\delta-1}{2-2^\delta}<\frac{\delta}{1-\delta}.
$$
And then, assuming $\delta<\frac12$, $\pi(h(N)<4$ and we have 
$$
\delta g(N)N<4.
$$
For the second term, we have by positivity
\begin{eqnarray*}
\int_1^{N}\frac{\pi(h(t))}{t(t+1)}[t]dt&<&\int_1^{\frac1{2^\delta-1}}\frac{\pi(h(t))}{(t+1)}dt=\delta\int_2^{2^{\frac{1}{\delta}}}\frac{\pi(u)}{(u^\delta-1)u}du\\
&<&\int_2^{2^{\frac{1}{\delta}}}\left(\frac{1}{(\log u)^2}+\frac32\frac{1}{(\log u)^3}\right)du\\
&=&\left.\frac{u}{(\log u)^2}\right|_2^{2^{\frac1\delta}}+\frac72\int_2^{2^{\frac{1}{\delta}}}\frac{1}{(\log u)^3}du
\end{eqnarray*}
where we have done the change of variables $u=h(t)$, and the mean value theorem for the function $u^x$, together with the bound (see \cite{roscho}) 
\begin{equation}\label{eq:boundpi}
\pi(u)<\frac{u}{\log u}+\frac32\frac{u}{(\log u)^2}
\end{equation} 
To bound the last integral, we again perform integration by parts, getting
\begin{equation}\label{eq:interror}
\int_2^{2^{\frac{1}{\delta}}}\frac{1}{(\log u)^3}du=\left.\frac u{(\log u)^3}\right|_2^{2^\frac1\delta}+3\int_2^{2^{\frac{1}{\delta}}}\frac{1}{(\log u)^4}du.
\end{equation}
Now
$$
3\int_2^{2^{\frac{1}{\delta}}}\frac{1}{(\log u)^4}du=3\int_2^{e^6}\frac{1}{(\log u)^4}du+3\int_{e^6}^{2^{\frac{1}{\delta}}}\frac{1}{(\log u)^4}du<11+\frac12\int_{2}^{2^{\frac{1}{\delta}}}\frac{1}{(\log u)^3}du
$$
where we have use Maple to evaluate the first integral. Now, pluging this into (\ref{eq:interror}) we get
$$
\int_2^{2^{\frac{1}{\delta}}}\frac{1}{(\log u)^3}du<2\delta^3\frac {2^{\frac1\delta}}{(\log 2)^3}-\frac{4}{(\log 2)^3}+22<2\delta^3\frac {2^{\frac1\delta}}{(\log 2)^3},
$$
and collecting all these estimates, we finally obtain
$$
\int_1^{N}\frac{\pi(h(t))}{t(t+1)}[t]dt<\delta^2\frac {2^{\frac1\delta}}{(\log 2)^2}+7\delta^3\frac {2^{\frac1\delta}}{(\log 2)^3}-4,
$$
which gives 
$$
\log(C(\delta))<H(\delta),
$$
as desired.
\end{proof}

{\bf Remark:} A result  weaker than this can be found in \cite{hawri} where, by ementary methods, they get $\tau(n)< e^{\frac{2^{1/\delta}}{\delta}}n^{\delta}$. Also,
it is important to note that the bound in \cite{r} is attained at certain integers, so
the content of Proposition \ref{prop:tau} is also of the right order of magnitude.
%$L=c\delta^32^{\frac1\delta}$. Then,  $\log L=\log c+3\log\delta+\frac1\delta\log 2> \frac1\delta\log 2$ for any  $c>\delta^{-3}$ and 
%$$
%\int_2^{2^{\frac{1}{\delta}}}\frac{1}{(\log u)^3}du=\int_2^{L}\frac{1}{(\log u)^3}du+\int_L^{2^{\frac{1}{\delta}}}<\frac{1}{(\log u)^3}du<\frac1{(\log 2)^3}L+\frac{2^{\frac1\delta}}{(\log L)^3}<\frac{(c+1)\delta^3}{(\log2)^3}2^{\frac1\delta}
%$$

\

\begin{proof}[Proof of Theorem \ref{teo:tau}] The bulk of the proof of  Theorem \ref{teo:tau} is an immediate consequence of 
Proposition \ref{prop:tau}. Indeed,  taking  $\delta=\frac{\log 2}{\log\log n}$ in Proposition {\ref{prop:tau}}, gives
\begin{equation}\label{eq:taupsi}
\tau(n)<n^{\left(\frac{\log 2}{\log\log n}+\frac{1}{(\log\log n)^2}+\frac{7}{(\log\log n)^3}\right)}.
\end{equation}
Now, the equation $7x^2+x-\varepsilon$ has one negative root and  one positive root at $x_0=\frac{\sqrt{1+28\varepsilon}-1}{14}$ and so, for any $0\le x<x_0$ the equation is negative and hence taking  $x=\frac1{\log\log n}$ in (\ref{eq:taupsi}) we get for any $n>e^{e^{\frac1{x_0}}}$
$$
\tau(n)<n^{\frac{\log 2+\varepsilon}{\log\log n}}.
$$
Observe that in fact (\ref{eq:taupsi}) is much better  for $n$ large enough than the second inequality in the theorem. Hence, this second result does not pretend to be precise, but rather a simple example  to state and apply in potential applications. Now, let us denote $\mu\approx 3.549$ such that $\mu^2+\mu=(3-\log 2)7$ or, in other words 
\begin{equation}\label{eq:mu}
\frac{1+\mu}{3-\log 2}=\frac7\mu. 
\end{equation}
Now, if  $n\ge 1321$ then $\log\log n\ge\frac{7}{\mu}$, and  we have
$$
\frac7{(\log\log n)^3}<\frac{\mu}{(\log\log n)^2}.
$$
But then
$$
e^{H(\delta)}<n^{\frac{1}{(\log\log n)^2}+\frac{7}{(\log\log n)^3}}<n^{\frac{1+\mu}{(\log\log n)^2}}\le n^{\frac{3-\log 2}{\log\log n}},
$$
by using again $\log\log n\ge\frac{7}{\mu}=\frac{1+\mu}{3-\log 2}$. The result follows in this case by just plugging this into Proposition \ref{prop:tau}. 

\

Finally if $n\le 1320$, the result is trivial. It can be confirmed in an instant with Maple.
\end{proof}

\section{Ending of the proof of Theorems \ref{repsqfree} and \ref{primrepsqfree}.}

To proceed with the proof of Theorem \ref{repsqfree}, we need to choose $\delta$ such that 
$$
E(n)<n^{1/2}\prod_{p}\left(1-\frac{g(p)}{p^2}\right).
$$
for $n$ as small as possible. 

\

Taking $\varepsilon=0.342$ in Theorem \ref{teo:tau}, we see that 
$$
E(n)<160{(\log n)^{\frac43}}n^{\frac13+\frac{\log 2+0.342}{\log\log n}}<4.9\times10^{218}
$$
for $n= 10^{440}$. On the other hand for the same $n$,  using (\ref{eq:gipi}) we get
$$
n^{\frac12}\prod_p\left(1-\frac{g(p)}{p^2}\right)>n^{\frac12}\frac{\prod_p\left(1-\frac{2}{p^2}\right)}{e^\gamma\log\log n+3}>6\times10^{218},
$$
by using $\kappa=\prod_p\left(1-\frac{2}{p^2}\right)>0.3226$, computed with Maple. This not only concludes the proof Theorem \ref{repsqfree}, but  also  provides a proof  that $r(n)>0$ for any $n\ge 10^{440}$, simply noting that 
$$
\frac{n^{\frac12}\prod_p\left(1-\frac{g(p)}{p^2}\right)}{E(n)}> 
\frac{\kappa n^{\frac16-\frac{\log 2+0.342}{\log\log n}}}{(160e^\gamma\log\log n+480)(\log n)^{\frac43}}>1
$$
for  $n\ge10^{440}$, since  we easily see with maple that $
 (160e^\gamma\log\log n+480)(\log n)^{\frac43}<(\log n)^{\frac73}$
 for $n\ge 10^{176}$, while $ \kappa n^{\frac16-\frac{\log 2+0.342}{\log\log n}}>(\log n)^{\frac73}$
 for $n\ge10^{440}$.  \hfil$\hskip160pt\square$
%Minimizing the function $f(\delta)=\delta^2\frac {2^{\frac1\delta}}{(\log 2)^2}+7\delta^3\frac {2^{\frac1\delta}}{(\log 2)^3}$,  we get for $\delta_0=0.2548$, the bound $f(\delta_0)=7.332$, which gives
%$$
%C(\delta)<e^{9.446}.
%$$

%Descarting the logarithmic terms, $r(n)$ will be positive for any $n$ of size close to 
%$$
%n(\delta)=e^{\frac{2\delta2^{1/\delta}}{1/6-\delta}}
%$$ 
%and so,  in order to choose $\delta$ properly, we want to minimize this function. It turns out that a selection of $\delta=0.134$ is optimal proving not only Theorem \ref{repsqfree} but also  proving that $r(n)>0$ for any $n\ge 10^{629}$.

\

{\bf Remark:} As we mentioned, the proof of Theorem \ref{primrepsqfree} is very similar but, in this case,  we  consider $S^*(n)=\{0\le x\le n^{1/2}: (x,n)=1\,,\,n-x^{2} \text{ is squarefree}\}$ and, again, write $r^*(n)=|S^*(n)|$ in terms of the m\"obius function. The factor  $\prod_{p|n}\left(1-\frac1p\right)$ of $r^*(n)$ in Theorem \ref{primrepsqfree} comes
from the extra condition $(x,n)=1$.
\section{Proof of Theorem \ref{sqfree}.}
It might be surprising, after all the computations above, but proving  Theorem \ref{sqfree}  is extremely simple. The idea is that 
to prove the theorem it is enough to prove positivity of the function $r^*(n)$ and, for that, we only need a lower bound. Hence, in order to avoid the divisor function, we restrict ourselves to prime numbers.  In fact the theorem follows from  the trivial inequality
$$
[\sqrt{n}]-r^*(n)<\sum_{\substack{p\le \sqrt{n}\\p\nmid n}}\sum_{\substack{0\le x\le \sqrt{n}\\ p^2|n-x^2}}1=\sum_{\substack{p\le \sqrt 2n^{\frac14}\\p\nmid n}}\sum_{{0\le x\le \sqrt{n}}\atop{p^2|n-x^2}}1+\sum_{\substack{\sqrt 2n^{\frac14}<p\le \sqrt{n}\\p\nmid n}}\sum_{{0\le x\le \sqrt{n}}\atop{p^2|n-x^2}}1=S_1+S_2
$$
To bound $S_1$, we note that  the number of solutions to the congruence $x^2\equiv n\tm {p^2}$ is bounded by $2$ for any $p\nmid n$ and, then
\begin{equation}\label{eq:errorp}
\sum_{\substack{p\le \sqrt{2}n^{\frac14}\\p\nmid n}}\sum_{{0\le x\le \sqrt{n}}\atop{p^2|n-x^2}}1\le 2\sqrt{n}\sum_{\substack{p\le \sqrt{2}n^{\frac14}\\p\nmid n}}\frac1{p^2}+ 2\sum_{\substack{p\le \sqrt{2}n^{\frac14}\\p\nmid n}}1.
\end{equation}
Now a trivial computation with maple gives
$$
\sum_{p\le 10^5}\frac1{p^2}=0.45223
$$
meanwhile for any $n\ge 1$
$$
\sum_{10^5<p\le \sqrt{2}n^{\frac14}}\frac1{p^2}<\int_{10^5}^{\infty}\frac{1}{t^2}dt=10^{-
5}.
$$
This bounds the first term of (\ref{eq:errorp}). For the second term we use (\ref{eq:boundpi}), and the two estimates together give 
\begin{equation}\label{eq:errorp2}
S_1< 0.9045\sqrt n+\frac{12\sqrt2n^{\frac14}}{\log n},
\end{equation}
for $n\ge 21$. 

\

To estimate $S_2$ we notice that, for each prime $p>\sqrt 2 n^{\frac14}$, from the two positive solutions  $x_0,x_1$ of $x^2\equiv n\pmod {p^2}$,  only one, say $x_0$, can verify  $x_0\le \sqrt n$ since $x_1=p^2-x_0>\sqrt n$. Hence, 
$$
S_2\le \sum_{\sqrt 2n^{\frac14}<p\le \sqrt{n}}1<\pi(\sqrt n)<\frac{2\sqrt n}{\log n}+\frac{6\sqrt n}{(\log n)^2}
$$
Adding the two estimations we get
$$
[\sqrt{n}]-r^*(n)<\sqrt n\left( 0.9045+\frac{2}{\log n}+\frac{6}{(\log n)^2}\right)+\frac{12\sqrt2n^{\frac14}}{\log n},
$$
which gives $r(n)>0$ for any $n\ge 179\times 10^{8}$. This number does not seems out of reach of modern computers. However, a simple observation will make much more accesible our final computations. Indded, we note that the problem comes from the very small primes, so we will continue distinguising cases by the primes $2$ and $3$. In particular, if $n\not\equiv 1\pmod 4$, then in $S_1$ we can avoid the prime $2$ since either $n$ is even, and then $2$ is not counted, or $n$ is not a square modulo $4$, and then there are no any $x$ solution to the equation $4|n-x^2$. This gives
$$
[\sqrt{n}]-r^*(n)<\sqrt n\left( 0.4045+\frac{2}{\log n}+\frac{6}{(\log n)^2}\right)+\frac{12\sqrt2n^{\frac14}}{\log n},
$$
and hence $r(n)>0$, for any $n\ge 200$. If $n\not\equiv\square\pmod 9$, then we proceed in a similar way, but removing now the prime $3$ from the sum, and we get
$$
[\sqrt{n}]-r^*(n)<\sqrt n\left( 0.6823+\frac{2}{\log n}+\frac{6}{(\log n)^2}\right)+\frac{12\sqrt2n^{\frac14}}{\log n},
$$
and positivity of $r(n)$ for any $n\ge 74249$. Finally, if $n\equiv \square\pmod {36}$ , then we have removed twice the solutions to $x^2\equiv n\pmod{36}$, and then since there are $4$ solutions, on each interval of lenght $36$, we get 
$$
[\sqrt{n}]-r^*(n)<S_1+S_2-\sum_{\substack{0\le x\le \sqrt{n}\\ 36|n-x^2}}1<\sqrt n\left( 0.7934+\frac{2}{\log n}+\frac{6}{(\log n)^2}\right)+\frac{12\sqrt2n^{\frac14}}{\log n}+4,
$$ 
which again gives $r(n)$  positive for any $n\ge 1375077$. 

\

To check the remaining cases, $n\le 1375077$, we just used the instruction issqrfree of Maple to confirm that $r(n)>0$ for any integer $n\le 1375077$ in less than one hour. It can be done in other ways, not using the implicit definition of maple, in similar amount of time.
 
\

\bibliographystyle{plain}
\bibliography{square-free-vii}

%\Addresses
\end{document}